\documentclass[11pt]{article}
\usepackage[a4paper,margin=1in]{geometry}
\usepackage[T1]{fontenc}
\usepackage{lmodern}
\usepackage{authblk}
\usepackage{amsmath,amssymb,amsthm,mathtools,mathrsfs}
\usepackage{enumitem,microtype,booktabs}
\usepackage{xcolor}
\usepackage{float,needspace}
\usepackage{tikz}
\usepackage{graphicx}
\usetikzlibrary{arrows.meta}
\usepackage[colorlinks=true,linkcolor=blue,citecolor=blue,urlcolor=blue]{hyperref}
\hypersetup{pdftitle={Bi-parameter local linearization for the stochastic wave equation with rough noise}}
\numberwithin{equation}{section}
\newtheorem{theorem}{Theorem}[section]
\newtheorem{proposition}[theorem]{Proposition}
\newtheorem{lemma}[theorem]{Lemma}

\theoremstyle{definition}
\newtheorem{condition}[theorem]{Condition}
\theoremstyle{remark}
\newtheorem{remark}[theorem]{Remark}
\newcommand{\R}{\mathbb R}
\newcommand{\E}{\mathbb E}
\newcommand{\F}{\mathcal F}
\newcommand{\Hh}{\mathcal H}
\newcommand{\1}{\mathbf 1}

\newcommand{\dd}{\mathrm d}

\newcommand{\norm}[1]{\left\lVert#1\right\rVert}

\newcommand{\Nn}{\mathcal N}
\newcommand{\energy}[1]{\mathscr E_H(#1)}
\newcommand{\energyp}[2]{\mathscr E_{H,#1}(#2)}
\allowdisplaybreaks[1]
\setlist[enumerate]{itemsep=3pt,topsep=5pt}
\title{{Bi-parameter local linearization  for the stochastic wave equation with rough noise}}
 \author[1]{Ruinan Li}
\author[2]{Ran Wang}

\affil[1]{School of Statistics and Data Science,
Shanghai University of International Business and Economics,
Shanghai 201620, China}

\affil[2]{School of Mathematics and Statistics,
Wuhan University,
Wuhan 430072, China}

\affil[ ]{\texttt{ruinanli@amss.ac.cn} \qquad
\texttt{rwang@whu.edu.cn}}

\date{}

\begin{document}
\maketitle

\begin{abstract}  
We study a one-dimensional nonlinear stochastic wave equation driven
by Gaussian noise that is white in time and rough in space.
We prove a bi-parameter local linearization for mixed increments along
the two characteristic directions. The proof combines characteristic
cancellation with a localized fractional-energy estimate that controls
boundary interactions caused by the rough spatial noise.
As an application, we establish quadratic-variation limits on arbitrary
anisotropic rectangular meshes and construct a consistent estimator
of a multiplicative diffusion parameter. Numerical experiments
illustrate the finite-sample performance of the estimator.
\end{abstract}
 
\noindent\textbf{Keywords.} {Stochastic wave equation; spatially rough noise; local linearization;  quadratic variation.}

\noindent\textbf{MSC 2020.} 60H15; 60G17; 60G22.

\section{Introduction and main results}\label{sec:introduction}

 We consider the one-dimensional nonlinear stochastic wave equation 
\begin{equation}\label{eq:SWE}
\begin{cases}
\partial_t^2u(t,x)=\partial_x^2u(t,x)+F(u(t,x))\dot W(t,x),
&t>0,\ x\in\R,\\
u(0,x)=u_0(x),\qquad \partial_tu(0,x)=v_0(x).
\end{cases}
\end{equation}
Here $\dot W=\partial_t\partial_xW$ in the distributional
sense, with $W$ being a centered Gaussian field with covariance 
\begin{equation}\label{eq:covW}
\E[W(t,x)W(s,y)]
=\frac{t\wedge s}{2}\big(|x|^{2H}+|y|^{2H}-|x-y|^{2H}\big),
\qquad \frac14<H<\frac12.
\end{equation}

The stochastic-integration and solution theory for spatially homogeneous wave and heat equations has foundations in~\cite{Dalang,PeszatZabczyk, Walsh}. For fractional spatial noise with $H\in(1/4,1/2)$, Balan et al. \cite{BalanJolisQuer, BalanJolisQuer16} treated heat and wave equations with affine multiplicative coefficients, while Hu et al.  \cite{HuHuangLeNualartTindel} developed the nonlinear heat equation theory.    Liu et al.~\cite{LHW} established existence and uniqueness of the mild solution to the nonlinear stochastic wave equation under general conditions.

 For space-time white noise, corresponding to $H=\frac{1}{2}$,
Khoshnevisan et al.~\cite{KSXZ2014} established a temporal local
linearization of the form
$$
X(t+h,x)-X(t,x)
\approx
F(X(t,x))\bigl[Y(t+h,x)-Y(t,x)\bigr],
\qquad h\downarrow0,
$$
where $Y$ denotes the solution to the additive SHE
(that is, $F\equiv1$) driven by the same noise. Hairer and Pardoux~\cite{HP15} obtained a sharper local description
under stronger smoothness assumptions, using a different approach.  
For spatial increments, Foondun et al.~\cite{FoondunKhoshnevisanMahboubi} established
the analogous local spatial linearization
$$
u(t,x+h)-u(t,x)
\approx
\sigma(u(t,x))\bigl[U(t,x+h)-U(t,x)\bigr],
\qquad h\downarrow0.
$$
 These approximations separate the local effect of the nonlinearity,
represented by the coefficient \(F(X(t,x))\), from the Gaussian
fluctuations, described by the increments of \(Y\).  Further developments and applications of this
approach can be found in
\cite{Das2022,HuangKhoshnevisanParabolic,KKM23, KhoshnevisanLeePuXiao2025,Wang2024,WangXiao2024}.

For the wave equation, finite-speed propagation creates a different localization problem.  Huang and Khoshnevisan~\cite{HuangKhoshnevisan} made this contrast with parabolic SPDEs precise for a nonlinear wave equation with space-time white noise: their quadratic-variation and fluctuation limits retain the diffusion coefficient over extended characteristic regions. Thus ordinary one-parameter increments do not, in general, admit the usual parabolic approximation obtained by freezing the coefficient only at the observation point.

 In the space-time white-noise case $H=1/2$, Huang, Oh and
Okamoto~\cite{HuangOhOkamoto} resolved this issue by considering
mixed second-order increments along the two null directions.
The four cone contributions cancel outside a small characteristic
region, localizing the stochastic integral near the base point
and allowing the diffusion coefficient to be frozen there.
Rang and Wang~\cite{RangWang} obtained a related local
linearization result for the nonlinear damped stochastic
Klein--Gordon equation.

Liu and Wang~\cite{LiuWang} extended the characteristic
mixed-increment approach to noise that is white in time and
spatially homogeneous, with fractional Riesz-type covariance
and Hurst parameter $H\in(1/2,1)$. They used the local linearization to study quadratic variations and construct
a consistent estimator of the diffusion parameter.

The present paper establishes bi-parameter local linearization
in the rough regime $H\in(1/4,1/2)$.
The key analytic input is a localized fractional-energy estimate
that controls both the variation inside each spatial section
and the interactions created by zero extension across its
boundary. Combined with characteristic cancellation, this
estimate yields a remainder bound uniform over all aspect ratios.

 We first specify the assumptions on the coefficient and the
initial data. Let
\begin{equation}\label{eq:I0}
I_0(t,x)
=
\frac12\big[u_0(x+t)+u_0(x-t)\big]
+\frac12\int_{x-t}^{x+t}v_0(y)\dd y
\end{equation}
denote the deterministic wave evolution associated with the
initial data in~\eqref{eq:SWE}.

Writing $$\mathfrak D_h f(t,x)=f(t,x+h)-f(t,x),$$ we use the
space $\mathcal Z^q(T)$ introduced in~\cite[(2.6)]{LHW}.
For $T>0$ and $2\leq q<\infty$, this space consists of
continuous maps $f:[0,T]\times\R\to L^q(\Omega)$ with finite norm
\begin{equation}\label{eq:Zp}
\begin{aligned}
\norm{f}_{\mathcal Z^q(T)}
={}&
\sup_{t\in[0,T]}
\norm{f(t,\cdot)}_{L^q(\Omega\times\R)}\\
&+
\sup_{t\in[0,T]}
\left(
\int_\R
\norm{\mathfrak D_h f(t,\cdot)}_{L^q(\Omega\times\R)}^2
|h|^{2H-2}\dd h
\right)^{1/2}.
\end{aligned}
\end{equation}

\begin{condition}\label{cond:main}
Throughout the paper, we impose the following assumptions.
\begin{enumerate}[label=\textnormal{(A\arabic*)}]
\item
$F\in C^1(\R)$, $F(0)=0$, and $F'$ is bounded and globally
Lipschitz.

\item
There exists $p_0>2/(4H-1)$ such that
$I_0\in\mathcal Z^{p_0}(T)$ for every $T>0$.
Moreover, for every $T>0$ there exists $K_0(T)<\infty$ such that
$$
\sup_{(t,x)\in[0,T]\times\R}|I_0(t,x)|\leq K_0(T)
$$
and
$$
|I_0(t,x)-I_0(s,y)|
\leq
K_0(T)\big(|t-s|^H+|x-y|^H\big)
$$
for all $s,t\in[0,T]$ and $x,y\in\R$.
\end{enumerate}
\end{condition}

Under Condition~\ref{cond:main}, let $u$ denote the mild
solution of~\eqref{eq:SWE}, and let $U$ denote the corresponding
additive-noise solution with unit diffusion coefficient and
zero initial data, driven by the same noise $W$.
In null coordinates, set
$$
\begin{aligned}
v(\tau,\lambda)
&=u\left(\frac{\tau+\lambda}{\sqrt2},
         \frac{\lambda-\tau}{\sqrt2}\right),\\
V(\tau,\lambda)
&=U\left(\frac{\tau+\lambda}{\sqrt2},
         \frac{\lambda-\tau}{\sqrt2}\right).
\end{aligned}
$$
For $h_1,h_2>0$, define the mixed increment
\begin{equation}\label{eq:mixed-definition}
\begin{aligned}
D_{h_1,h_2}f(\tau,\lambda)
={}&f(\tau+h_1,\lambda+h_2)-f(\tau+h_1,\lambda)\\
&-f(\tau,\lambda+h_2)+f(\tau,\lambda).
\end{aligned}
\end{equation}
Write $$z=(\tau,\lambda),\ t_z=(\tau+\lambda)/\sqrt2,\
m=h_1\wedge h_2, \ M=h_1\vee h_2.$$
The increment is admissible on $[0,T]$ if
$0\leq t_z$ and $t_z+(h_1+h_2)/\sqrt2\leq T$.

Define
$$
\Psi_H(h_1,h_2)
=
2^{H-5/2}\left[(M-m)m^{2H}
+\frac{2m^{2H+1}}{2H+1}\right]
\asymp m^{2H}M.
$$
Proposition~\ref{prop:variance} identifies $\Psi_H(h_1,h_2)$
as the variance of $D_{h_1,h_2}V(z)$.

\begin{theorem}[Bi-parameter local linearization]\label{thm:main}
{Assume Condition~\ref{cond:main}. For every $T>0$ and finite $p\geq2$, there exists $C=C(T,p,H,F,I_0)$ such that every admissible positive increment with $M\leq1$ satisfies}
\begin{equation}\label{eq:main}
{\displaystyle
\left\|D_{h_1,h_2}v(z)-F(v(z))D_{h_1,h_2}V(z)\right\|_{L^p(\Omega)}
\leq C M^H\Psi_H(h_1,h_2)^{1/2}.
}
\end{equation}
The constant is uniform in $z,h_1$, and $h_2$; in particular,
it is independent of the aspect ratio $h_1/h_2$.
\end{theorem}

 The local-linearization theorem provides a Gaussian approximation
for mixed increments, which naturally suggests quadratic-variation
statistics based on their squares and, in turn, diffusion-parameter
estimation from discrete observations. For linear stochastic wave
equations, Khalil et al.~\cite{KhalilTudorZili} studied spatial
quadratic variations under additive space-time white noise, while
Khalil and Tudor~\cite{KhalilTudor} analyzed spatial quadratic
variations and constructed a consistent estimator of the Hurst
parameter for noise that is fractional in time and white in space.

An additional feature of the rough regime emerges in the
quadratic-variation analysis. Since $H<1/2$, the Gaussian mixed
increments associated with distinct grid cells are negatively
correlated. Combined with the rectangular partition structure,
this yields a uniform covariance row bound and removes any
aspect-ratio restriction from the weighted Gaussian array estimate.
Consequently, the rectangular quadratic variations converge in
$L^p(\Omega)$ for arbitrary anisotropic meshes, leading in turn
to a consistent estimator of the diffusion multiplier.

 The remainder of the paper is organized as follows.
Section~\ref{sec:preliminaries} develops the analytic estimates
needed for the rough-noise regime, including wave-kernel bounds and a localized
fractional-energy estimate. Section~\ref{sec:geometry} analyzes
the characteristic geometry of mixed null-coordinate increments,
derives the exact variance of the associated Gaussian increment,
and proves the bi-parameter local linearization theorem.
  Finally, Section~\ref{sec:statistics} studies anisotropic rectangular
quadratic variations and diffusion-parameter estimation, and concludes
with a numerical illustration of the estimator.

\section{{Rough-noise estimates}}\label{sec:preliminaries}

\subsection{Rough-noise framework}\label{subsec:framework}

We follow the stochastic integration framework of Liu
et al.~\cite[Section~2]{LHW}. Let $\Hh_H$ denote the completion
of $C_c^\infty(\R)$ under the norm induced by the inner product
\begin{equation}\label{eq:spatial-inner}
\langle f,g\rangle_{\Hh_H}
=
a_H\int_{\R^2}
[f(y)-f(z)][g(y)-g(z)]|y-z|^{2H-2}\dd y\dd z,
\end{equation}
where
\begin{equation}\label{eq:aH}
a_H=H\left(\frac12-H\right).
\end{equation}
The noise $W$ extends to an isonormal Gaussian process over
the space-time Hilbert space $\Hh_T=L^2([0,T];\Hh_H)$.
With this normalization, the indicator of every finite interval
$(a,b)$ belongs to $\Hh_H$ and satisfies
\begin{equation}\label{eq:indicator-norm}
\norm{\1_{(a,b)}}_{\Hh_H}^2=|b-a|^{2H}.
\end{equation}

Stochastic integration with respect to $W$ is defined for
admissible predictable $\Hh_H$-valued processes as
in~\cite[Definition~2.1 and Proposition~2.2]{LHW}.  Here,  the
predictability is understood with respect to the usual
augmentation $(\mathcal F_t)_{t\geq0}$ of the natural
filtration generated by $W$.
For such a process with a jointly predictable random-field
representative $g$, the moment estimate
in~\cite[Proposition~2.3, (2.5)]{LHW} gives, for
$2\leq p<\infty$,
\begin{equation}\label{eq:BDG}
\left\|\int_0^T\int_\R g(s,y)\,W(\dd s,\dd y)\right\|_{L^p(\Omega)}^2
\leq C_{p,H}\int_0^T\int_{\R^2}
\norm{\mathfrak D_h g(s,y)}_{L^p(\Omega)}^2
|h|^{2H-2}\dd y\dd h\dd s.
\end{equation}

We use the pointwise seminorm
introduced in~\cite[(2.4)]{LHW}, 
\begin{equation}\label{eq:Npoint}
\Nn_{\frac12-H, p}f(t,x)
=
\left(
\int_\R
\norm{\mathfrak D_h f(t,x)}_{L^p(\Omega)}^2
|h|^{2H-2}\dd h
\right)^{1/2}.
\end{equation}
 
 We next record a consequence of
Condition~\ref{cond:main}\textnormal{(A2)} that allows us to
use arbitrarily large finite moment orders.
The boundedness and H\"older continuity of $I_0$ imply
\begin{equation}\label{eq:I0-bounds}
\begin{aligned}
\sup_{t\in[0,T]}
\norm{I_0(t,\cdot)}_{L^\infty(\R)}
&\leq K_0(T),\\
\sup_{t\in[0,T]}
\norm{\mathfrak D_h I_0(t,\cdot)}_{L^\infty(\R)}
&\leq 2K_0(T)(|h|^H\wedge1).
\end{aligned}
\end{equation} 
To obtain higher integrability, fix $q>p_0$ and set
$\vartheta=p_0/q\in(0,1)$.
  Interpolation between $L^{p_0}(\R)$ and $L^\infty(\R)$ gives
$$
\norm{\mathfrak D_h I_0(t,\cdot)}_{L^q(\R)}
\leq
\norm{\mathfrak D_h I_0(t,\cdot)}_{L^{p_0}(\R)}^\vartheta
\big[2K_0(T)(|h|^H\wedge1)\big]^{1-\vartheta},
$$
with an analogous estimate for $I_0(t,\cdot)$ itself.
Since $H\in(1/4,1/2)$,
$$
\int_\R (|h|^{2H}\wedge1)|h|^{2H-2}\dd h
=
\frac{2}{4H-1}+\frac{2}{1-2H}
<\infty.
$$
Applying H\"older's inequality  with respect to the
measure $|h|^{2H-2}\dd h$, and taking suprema over
$t\in[0,T]$, therefore yields
$$
\norm{I_0}_{\mathcal Z^q(T)}
\leq
C_H K_0(T)^{1-\vartheta}
\norm{I_0}_{\mathcal Z^{p_0}(T)}^\vartheta.
$$
The required continuity follows from~\textnormal{(A2)}.
Together with the assumed case $q=p_0$, this proves that
$I_0\in\mathcal Z^q(T)$ for every  $q\geq p_0$.
 
 The one-dimensional wave kernel is given by
\begin{equation}\label{eq:kernel-I0}
G_t(x)
 =\frac12\1_{\{|x|<t\}},\qquad t>0.
\end{equation}
With $I_0$ defined in~\eqref{eq:I0}, we interpret
\eqref{eq:SWE} in the mild sense: 
$$
u(t,x)
=
I_0(t,x)
+\int_0^t\int_\R
G_{t-s}(x-y)F(u(s,y))\,W(\dd s,\dd y).
$$ 
 
 Combining \cite[Theorem~2.6]{LHW} with the uniform moment
estimates in \cite[Proposition~3.3]{LHW} yields the following
result. Indeed, for every finite $q\geq p_0$,
Condition~\ref{cond:main}\textnormal{(A2)} and the preceding
interpolation imply $I_0\in\mathcal Z^q(T)$, while
\textnormal{(A2)} also gives
$$
\sup_{(t,x)\in[0,T]\times\R}
\mathcal N_{\frac12-H,q}I_0(t,x)<\infty.
$$
The solutions obtained at different moment orders are identified
by pathwise uniqueness. For $2\leq p<p_0$, the corresponding
bounds follow by monotonicity of the $L^p(\Omega)$ norms.

\begin{theorem}\label{thm:wellposedness}
Assume Condition~\ref{cond:main}. For every $T>0$,
equation~\eqref{eq:SWE} admits a unique   
mild solution $u$ on $[0,T]\times\R$ with continuous sample
paths. Moreover, for every $2\leq p<\infty$,
$$
M_p(T)
:=
\sup_{(t,x)\in[0,T]\times\R}
\left(
\norm{u(t,x)}_{L^p(\Omega)}
+\Nn_{\frac12-H, p}u(t,x)
\right)
<\infty.
$$
\end{theorem}

The corresponding additive-noise solution with zero initial
data is the stochastic convolution
$$
U(t,x)
=
\int_0^t\int_\R
G_{t-s}(x-y)\,W(\dd s,\dd y).
$$
It is defined on the same probability space and driven by
the same noise $W$ as $u$.

\subsection{Kernel estimates and pointwise moment regularity}
\label{subsec:regularity}
 
For a deterministic function $k\in L^2(\R)$, define its possibly
infinite fractional energy by
$$
\energy{k}
=
\int_{\R^2}
|\mathfrak D_hk(y)|^2|h|^{2H-2}\dd y\dd h.
$$
For every finite interval $I=(a,b)$, the normalization
in~\eqref{eq:spatial-inner} and~\eqref{eq:indicator-norm} gives
\begin{equation}\label{eq:indicator-energy}
\energy{\1_I}
=
a_H^{-1}\norm{\1_I}_{\Hh_H}^2
=
\frac{2}{H(1-2H)}|b-a|^{2H}.
\end{equation} 

We use the Fourier transform convention
$$
\widehat k(\xi)=\int_\R e^{-i\xi y}k(y)\dd y,
\qquad k\in L^1(\R)\cap L^2(\R),
$$
and extend the transform to $L^2(\R)$ by continuity.
With this convention, Plancherel's identity reads
$$
\norm{k}_{L^2(\R)}^2
=
\frac1{2\pi}\int_\R|\widehat k(\xi)|^2\dd\xi.
$$
Applying Plancherel's identity to $\mathfrak D_hk$ and then
using Tonelli's theorem, we obtain
\begin{equation}\label{eq:Fourier-energy}
\energy{k}
=
b_H\int_\R|\widehat k(\xi)|^2|\xi|^{1-2H}\dd\xi,
\end{equation}
where the equality holds in $[0,\infty]$, and 
$$
b_H=
\frac1{2\pi}\int_\R|e^{iz}-1|^2|z|^{2H-2}\dd z.
$$

The following lemma collects the spatial and temporal increment
estimates for the wave kernel, together with its integrated
$L^2$ norm and fractional energy near time zero.

\begin{lemma}\label{lem:kernel}
Let $T>0$ and $H\in(1/4,1/2)$. The following estimates hold.
\begin{itemize}
\item[(i)]
For every $0<|h|\leq1$,
\begin{equation}\label{eq:kernel-space1}
\int_0^T
\norm{\mathfrak D_hG_r}_{L^2(\R)}^2\dd r
\leq C_T|h|,
\end{equation}
and
\begin{equation}\label{eq:kernel-space2}
\int_0^T
\energy{\mathfrak D_hG_r}\dd r
\leq C_{T,H}|h|^{2H}.
\end{equation}

\item[(ii)]
For every $0<h\leq1$,
\begin{equation}\label{eq:kernel-time1}
\int_0^T
\norm{G_{r+h}-G_r}_{L^2(\R)}^2\dd r
\leq C_Th,
\end{equation}
and
\begin{equation}\label{eq:kernel-time2}
\int_0^T
\energy{G_{r+h}-G_r}\dd r
\leq C_{T,H}h^{2H}.
\end{equation}

\item[(iii)]
For every $0<h\leq1$,
\begin{equation}\label{eq:kernel-new}
\int_0^h
\norm{G_r}_{L^2(\R)}^2\dd r
=
\frac{h^2}{4},
\qquad
\int_0^h
\energy{G_r}\dd r
=
C_Hh^{2H+1},
\end{equation}
where
$$
C_H
=
\frac{2^{2H-1}}{H(1-2H)(2H+1)}.
$$
\end{itemize}
All constants are positive and independent of $h$.
\end{lemma}

  \begin{proof}
\noindent\textit{(i) Spatial increments.} For $r>0$ and $0<|h|\leq1$, the intervals $(-r,r)$ and
$(-r-h,r-h)$ have intersection length $(2r-|h|)_+$,
where $a_+=\max\{a,0\}$. Hence,
$$
\begin{aligned}
\norm{\mathfrak D_hG_r}_{L^2(\R)}^2
&=
\frac14\int_\R
\big[
\1_{(-r-h,r-h)}(x)-\1_{(-r,r)}(x)
\big]^2\dd x\\
&=
\frac14\big[4r-2(2r-|h|)_+\big]\\
&=
\frac12\min\{|h|,2r\}
\leq \frac{|h|}{2}.
\end{aligned}
$$
Integrating over $r\in(0,T)$ proves
\eqref{eq:kernel-space1}.

For the fractional-energy estimate, use
$$
\widehat G_r(\xi)=\frac{\sin(r\xi)}{\xi},
\qquad
\widehat{\mathfrak D_hG_r}(\xi)
=(e^{ih\xi}-1)\widehat G_r(\xi),
$$
with the values at $\xi=0$ understood by continuity.
Set
$$
J_H
:=
\int_\R |e^{i\eta}-1|^2|\eta|^{-1-2H}\dd\eta.
$$
Since $|e^{i\eta}-1|\leq\min\{|\eta|,2\}$,
$$
J_H
\leq
2\int_0^1\eta^{1-2H}\dd\eta
+
8\int_1^\infty\eta^{-1-2H}\dd\eta
<\infty.
$$
The Fourier representation~\eqref{eq:Fourier-energy},
Tonelli's theorem, and $\sin^2(r\xi)\leq1$ therefore give
$$
\begin{aligned}
\int_0^T\energy{\mathfrak D_hG_r}\dd r
&=
b_H\int_\R
|e^{ih\xi}-1|^2|\xi|^{-1-2H}
\left(\int_0^T\sin^2(r\xi)\dd r\right)\dd\xi\\
&\leq
b_HT\int_\R
|e^{ih\xi}-1|^2|\xi|^{-1-2H}\dd\xi\\
&=
b_HTJ_H|h|^{2H},
\end{aligned}
$$
where the last equality follows from $\eta=h\xi$.
This proves~\eqref{eq:kernel-space2}.

\noindent\textit{(ii) Temporal increments.}
For $h>0$,
$$
G_{r+h}-G_r
=
\frac12
\big[
\1_{(-r-h,-r)}+\1_{(r,r+h)}
\big]
\qquad\text{a.e.}
$$
The two intervals are disjoint and each has length $h$,
so
$$
\norm{G_{r+h}-G_r}_{L^2(\R)}^2=\frac h2,
\qquad
\int_0^T\norm{G_{r+h}-G_r}_{L^2(\R)}^2\dd r
=\frac{Th}{2}.
$$
This proves~\eqref{eq:kernel-time1}.

Moreover, the sine-difference identity yields
$$
|\sin((r+h)\xi)-\sin(r\xi)|^2
\leq
4\sin^2(h\xi/2)
=
|e^{ih\xi}-1|^2.
$$
Thus~\eqref{eq:Fourier-energy} and the scaling calculation
in~\textnormal{(i)} imply
$$
\begin{aligned}
\int_0^T\energy{G_{r+h}-G_r}\dd r
&\leq
b_HT\int_\R
|e^{ih\xi}-1|^2|\xi|^{-1-2H}\dd\xi\\
&=
b_HTJ_Hh^{2H},
\end{aligned}
$$
proving~\eqref{eq:kernel-time2}.

\noindent\textit{(iii) Integrals near time zero.}
The definition of $G_r$ and~\eqref{eq:indicator-energy} give
$$
\norm{G_r}_{L^2(\R)}^2=\frac r2,
\qquad
\energy{G_r}
=
\frac14\energy{\1_{(-r,r)}}
=
\frac{2^{2H-1}}{H(1-2H)}r^{2H}.
$$
Integrating these identities over $r\in(0,h)$ yields
$$
\int_0^h\norm{G_r}_{L^2(\R)}^2\dd r
=
\frac{h^2}{4},
\qquad
\int_0^h\energy{G_r}\dd r
=
\frac{2^{2H-1}h^{2H+1}}
{H(1-2H)(2H+1)},
$$
which proves~\eqref{eq:kernel-new}.
\end{proof}
 
 The following proposition can be obtained by adapting the
stochastic-convolution arguments of Liu et al.
 \cite[Section~4.1]{LHW} to pointwise $L^p(\Omega)$
estimates under the present assumptions. For completeness,
we provide a detailed proof, making explicit the exponent
$H$ needed below.
\begin{proposition}\label{prop:regularity}
For every $T>0$ and finite $p\geq2$, there exists $C=C(T,p,H,F,I_0)$ such that  
\begin{equation}\label{eq:exact-holder}
\norm{u(t,x)-u(s,y)}_{L^p(\Omega)}
\leq C\big(|t-s|^H+|x-y|^H\big)
\end{equation}
for all $s,t\in[0,T]$ and $x,y\in\R$.  
\end{proposition}
 
 \begin{proof}
Fix $T>0$ and $2\leq p<\infty$. Write $u=I_0+\Phi$,
where $\Phi$ is the stochastic convolution.  
Since $F(0)=0$ and $F$ is globally Lipschitz,
Theorem~\ref{thm:wellposedness} gives
$$
\sup_{(r,z)\in[0,T]\times\R}
\left(
\norm{F(u(r,z))}_{L^p(\Omega)}
+\Nn_{\frac12-H, p}(F(u))(r,z)
\right)
\leq C.
$$

For a deterministic measurable kernel $K$ with bounded
spatial support, the product identity
$$
\mathfrak D_\ell(KF(u))(r,z)
=
(\mathfrak D_\ell K)(r,z)F(u(r,z+\ell))
+
K(r,z)\mathfrak D_\ell(F(u))(r,z)
$$
and~\eqref{eq:Npoint} imply
$$
\begin{aligned}
&\int_{\R^2}
\norm{\mathfrak D_\ell(KF(u))(r,z)}_{L^p(\Omega)}^2
|\ell|^{2H-2}\dd z\dd\ell\\
 \leq &\,
C\energy{K(r,\cdot)}
+
2\int_\R |K(r,z)|^2
\big[\Nn_{\frac12-H, p}(F(u))(r,z)\big]^2\dd z\\
 \leq &\,
C\left(
\energy{K(r,\cdot)}
+\norm{K(r,\cdot)}_{L^2(\R)}^2
\right).
\end{aligned}
$$ 
Applying~\eqref{eq:BDG}, we therefore obtain
\begin{equation}\label{eq:kernel-product}
\left\|
\int_J\int_\R
K(r,z)F(u(r,z))\,W(\dd r,\dd z)
\right\|_{L^p(\Omega)}^2
\leq
C\int_J
\left(
\energy{K(r,\cdot)}
+\norm{K(r,\cdot)}_{L^2(\R)}^2
\right)\dd r.
\end{equation}

\noindent\textit{(i). Spatial increments.}
For $t\in[0,T]$, $x\in\R$, and $0<|h|\leq1$, apply
\eqref{eq:kernel-product} with $J=(0,t)$ and
$$
K(r,z)=G_{t-r}(x+h-z)-G_{t-r}(x-z).
$$
Translation and reflection preserve the kernel norms, so
\eqref{eq:kernel-space1} and~\eqref{eq:kernel-space2} yield
\begin{equation}\label{eq:Phi-space}
\norm{\mathfrak D_h\Phi(t,x)}_{L^p(\Omega)}^2
\leq
C\big(|h|^{2H}+|h|\big)
\leq C|h|^{2H},
\end{equation}
where   $2H<1$ is used in the last step.

\noindent\textit{(ii). Temporal increments.}
For $0<h\leq1$ and $0\leq t<t+h\leq T$, write
$$
\Phi(t+h,x)-\Phi(t,x)=B_1+B_2,
$$
where
$$
\begin{aligned}
B_1
&=
\int_0^t\int_\R
\big[G_{t+h-r}(x-z)-G_{t-r}(x-z)\big]
F(u(r,z))\,W(\dd r,\dd z),\\
B_2
&=
\int_t^{t+h}\int_\R
G_{t+h-r}(x-z)F(u(r,z))\,W(\dd r,\dd z).
\end{aligned}
$$
Applying~\eqref{eq:kernel-product} together with
\eqref{eq:kernel-time1} and~\eqref{eq:kernel-time2}
for $B_1$, and~\eqref{eq:kernel-new} for $B_2$, gives
$$
\begin{aligned}
\norm{B_1}_{L^p(\Omega)}^2
&\leq C(h^{2H}+h)\leq Ch^{2H},\\
\norm{B_2}_{L^p(\Omega)}^2
&\leq C(h^{2H+1}+h^2)\leq Ch^{2H+1}.
\end{aligned}
$$
Hence, by the triangle inequality,
\begin{equation}\label{eq:Phi-time}
\norm{\Phi(t+h,x)-\Phi(t,x)}_{L^p(\Omega)}
\leq
C\big(h^H+h^{H+1/2}\big)
\leq Ch^H.
\end{equation}

Combining~\eqref{eq:Phi-space} and~\eqref{eq:Phi-time},
using symmetry in the time variables and
Condition~\ref{cond:main}\textnormal{(A2)}, proves
\eqref{eq:exact-holder} when $|t-s|+|x-y|\leq1$.
For larger displacements,
$\max\{|t-s|,|x-y|\}>1/2$, and the uniform moment bound
in Theorem~\ref{thm:wellposedness} gives the same estimate.
\end{proof}

 \subsection{Localized fractional-energy estimate}
\label{subsec:energy}

We next estimate the fractional energy of a coefficient
localized to a finite interval.   

For $2\leq p<\infty$ and a   measurable map
$A:\R\to L^p(\Omega)$, define
$$
\energyp{p}{A}
=
\int_{\R^2}
\norm{A(y)-A(z)}_{L^p(\Omega)}^2
|y-z|^{2H-2}\dd y\dd z
\in[0,\infty].
$$ 

\begin{lemma}[Localized fractional energy]
\label{lem:localized-energy}
Let $2\leq p<\infty$, let $I=(a,b)$ have length
$\ell=b-a>0$, and let $A:I\to L^p(\Omega)$ be strongly
measurable. Suppose that, for finite constants $M_0,M_1\geq0$
and some $\gamma>1/2-H$,
$$
\sup_{y\in I}\norm{A(y)}_{L^p(\Omega)}\leq M_0,
$$
and
$$
\norm{A(y)-A(z)}_{L^p(\Omega)}
\leq M_1|y-z|^\gamma
\quad\text{for all }y,z\in I.
$$
Define the zero extension of $A$ by
$$
\widetilde A(y)
=
\begin{cases}
A(y), & y\in I,\\
0,    & y\notin I.
\end{cases}
$$
Then
\begin{equation}\label{eq:localized-energy}
\energyp{p}{\widetilde A}
\leq
\frac{2M_0^2}{H(1-2H)}\ell^{2H}
+
\frac{2M_1^2}
{(2H+2\gamma-1)(2H+2\gamma)}
\ell^{2H+2\gamma}.
\end{equation}
\end{lemma}

\begin{proof}
Since $\widetilde A$ vanishes on $I^c$, the contribution
from $I^c\times I^c$ is zero. By symmetry of the two
cross terms,
$$
\energyp{p}{\widetilde A}
=
I_{\mathrm{int}}+I_{\mathrm{bdry}},
$$
where
$$
\begin{aligned}
I_{\mathrm{int}}
&=
\int_{I^2}
\norm{A(y)-A(z)}_{L^p(\Omega)}^2
|y-z|^{2H-2}\dd y\dd z,\\
I_{\mathrm{bdry}}
&=
2\int_I\int_{I^c}
\norm{A(y)}_{L^p(\Omega)}^2
|y-z|^{2H-2}\dd z\dd y.
\end{aligned}
$$
 
Since $2H+2\gamma-2>-1$, the H\"older  continuity gives
$$
\begin{aligned}
I_{\mathrm{int}}
&\leq
2M_1^2\int_a^b\int_a^y
(y-z)^{2H+2\gamma-2}\dd z\dd y\\
&=
\frac{2M_1^2\ell^{2H+2\gamma}}
{(2H+2\gamma-1)(2H+2\gamma)}.
\end{aligned}
$$

For the boundary term, the amplitude bound and
\eqref{eq:indicator-energy} give
$$
I_{\mathrm{bdry}}
\leq M_0^2\energy{\1_I}
=
\frac{2M_0^2}{H(1-2H)}\ell^{2H}.
$$  
Adding the two estimates proves~\eqref{eq:localized-energy}.
\end{proof}
  
\section{{Anisotropic characteristic geometry and local linearization}}\label{sec:geometry}
  
 \subsection{The characteristic parallelogram}

Fix an admissible positive increment $(h_1,h_2)$ at
$z=(\tau,\lambda)$, corresponding to the point
$(t_0,x_0)$.   Set
$$
a=\frac{h_1}{\sqrt2},\qquad
b=\frac{h_2}{\sqrt2},\qquad
a_-=a\wedge b,\qquad
a_+=a\vee b.
$$
The four vertices in   physical coordinates are
$$
\begin{aligned}
P_0&=(t_0,x_0),&
P_1&=(t_0+a,x_0-a),\\
P_2&=(t_0+b,x_0+b),&
P_{12}&=(t_0+a+b,x_0+b-a).
\end{aligned}
$$
Define the characteristic parallelogram by
$$
\mathcal P_z(h_1,h_2)
=
\left\{
(q,y):
\tau<\frac{q-y}{\sqrt2}<\tau+h_1,\quad
\lambda<\frac{q+y}{\sqrt2}<\lambda+h_2
\right\}.
$$
It lies in the time strip $t_0<q<t_0+a+b$.
For $0<\rho<a+b$, its spatial section at time
$q=t_0+\rho$ is
\begin{equation}\label{eq:sections}
I_{t_0+\rho}
=
\big(
x_0+\max\{-\rho,\rho-2a\},\,
x_0+\min\{\rho,2b-\rho\}
\big).
\end{equation}
The section length is therefore
\begin{equation}\label{eq:section-profile}
\ell(\rho)
:=|I_{t_0+\rho}|
=
2\min\{\rho,a,b,a+b-\rho\},
\qquad 0<\rho<a+b.
\end{equation}

Figure~\ref{fig:parallelogram} summarizes the geometry of
$\mathcal P_z(h_1,h_2)$. The left panel shows the characteristic
parallelogram and its spatial section $I_{t_0+\rho}$ at time
$t_0+\rho$, while the right panel displays the corresponding
section-length profile $\ell(\rho)$. In particular, the left
panel makes transparent the bounds
$$
0<\rho<a+b,\qquad
|y-x_0|\leq a\vee b,\qquad
|I_{t_0+\rho}|\leq 2(a\wedge b),
$$
which will be used in the proof of
Theorem~\ref{thm:main}; the section-length profile will also be
used below in the exact variance calculation.

\begin{lemma}[Characteristic cancellation]\label{lem:geometry}
Let $z=(\tau,\lambda)$ and let $(h_1,h_2)$ be an
admissible positive increment. Then
\begin{equation}\label{eq:mixed-integrals}
\begin{aligned}
D_{h_1,h_2}v(z)
&=
\frac12
\int_{\mathcal P_z(h_1,h_2)}
F(u(q,y))\,W(\dd q,\dd y),\\
D_{h_1,h_2}V(z)
&=
\frac12
\int_{\mathcal P_z(h_1,h_2)}
W(\dd q,\dd y).
\end{aligned}
\end{equation}
Consequently, if $z$ corresponds to $(t_0,x_0)$, then
\begin{equation}\label{eq:mixed-remainder-representation}
\begin{aligned}
&D_{h_1,h_2}v(z)
-F(u(t_0,x_0))D_{h_1,h_2}V(z)\\
&\qquad=
\frac12
\int_{\mathcal P_z(h_1,h_2)}
\big[F(u(q,y))-F(u(t_0,x_0))\big]
\,W(\dd q,\dd y).
\end{aligned}
\end{equation}
\end{lemma}

\begin{proof}
In characteristic coordinates, the d'Alembert term is
the sum of a function of $\tau$ and a function of $\lambda$,
so its mixed increment vanishes.

For $P=(t,x)$, write
$$
\Delta(P)=\{(q,y):0<q<t,\ |x-y|<t-q\}
$$
for its backward light cone.
As illustrated in Figure~\ref{fig:characteristic-cancellation},
the four backward cones satisfy
$$
\1_{\Delta(P_{12})}
-\1_{\Delta(P_1)}
-\1_{\Delta(P_2)}
+\1_{\Delta(P_0)}
=
\1_{\mathcal P_z(h_1,h_2)}
\qquad\text{a.e.}
$$
The corresponding stochastic integrands are admissible
by the mild formulation, as is their finite signed sum.
The first identity in~\eqref{eq:mixed-integrals} thus follows
by linearity of the stochastic integral.
The same argument gives the second identity.

Finally, $F(u(t_0,x_0))$ is $\F_{t_0}$-measurable, and
$\mathcal P_z(h_1,h_2)$ lies strictly after $t_0$.
Together with the moment bounds in
Theorem~\ref{thm:wellposedness}
and~\eqref{eq:indicator-norm}, this shows that
$F(u(t_0,x_0))\1_{\mathcal P_z(h_1,h_2)}$
is a square-integrable predictable $\Hh_H$-valued integrand.
The pull-out property of the stochastic integral then yields
\eqref{eq:mixed-remainder-representation}.
\end{proof}

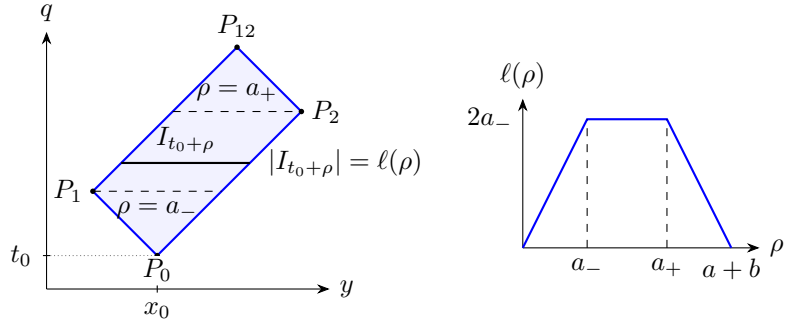
\begin{figure}[!htbp]
\centering
{\begin{tikzpicture}[scale=0.68,>=Stealth,font=\small]
\draw[->] (-2.15,-0.65)--(-2.15,4.35) node[above] {$q$};
\draw[->] (-2.15,-0.65)--(3.35,-0.65) node[right] {$y$};
\draw[gray,densely dotted] (-2.15,0)--(0,0)--(0,-0.65);
\draw (-2.22,0)--(-2.08,0) node[left,xshift=-3pt] {$t_0$};
\draw (0,-0.72)--(0,-0.58);
\node[below] at (0,-0.72) {$x_0$};
\fill[blue!5] (0,0)--(-1.25,1.25)--(1.55,4.05)--(2.8,2.8)--cycle;
\draw[thick,blue] (0,0)--(-1.25,1.25)--(1.55,4.05)--(2.8,2.8)--cycle;
\draw[dashed] (-1.25,1.25)--(1.25,1.25)
node[midway,below] {$\rho=a_-$};
\draw[dashed] (0.3,2.8)--(2.8,2.8)
node[midway,above] {$\rho=a_+$};
\draw[thick] (-0.7,1.8)--(1.8,1.8) node[midway,above] {$I_{t_0+\rho}$};
\node[right,xshift=3pt] at (1.8,1.8) {$|I_{t_0+\rho}|=\ell(\rho)$};
\fill (0,0) circle(1.5pt) node[below,fill=white,inner sep=1pt] {$P_0$};
\fill (-1.25,1.25) circle(1.5pt) node[left] {$P_1$};
\fill (2.8,2.8) circle(1.5pt) node[right] {$P_2$};
\fill (1.55,4.05) circle(1.5pt) node[above] {$P_{12}$};
\begin{scope}[xshift=7.1cm,yshift=0.15cm]
\draw[->] (0,0)--(4.6,0) node[right] {$\rho$};
\draw[->] (0,0)--(0,2.9) node[above] {$\ell(\rho)$};
\draw[thick,blue] (0,0)--(1.25,2.5)--(2.8,2.5)--(4.05,0);
\draw[dashed] (1.25,0)--(1.25,2.5);
\draw[dashed] (2.8,0)--(2.8,2.5);
\node[below] at (1.25,0) {$a_-$};
\node[below] at (2.8,0) {$a_+$};
\node[below] at (4.05,0) {$a+b$};
\node[left] at (0,2.5) {$2a_-$};
\end{scope}
\end{tikzpicture}}
\caption{A characteristic parallelogram with $h_1<h_2$ and its spatial-section profile. }
\label{fig:parallelogram}
\end{figure}

\begin{figure}[!htbp]
\centering
\begin{tikzpicture}[scale=0.62,>=Stealth,font=\small]
% Plotting coordinates are (y-x_0,q), with t_0=1.5, a=1.25 and b=2.8.
\coordinate (C0) at (0,1.5);
\coordinate (C1) at (-1.25,2.75);
\coordinate (C2) at (2.8,4.3);
\coordinate (C12) at (1.55,5.55);
\draw[->] (-4.6,0)--(7.6,0) node[right] {$y$};
\draw[->] (-4.6,0)--(-4.6,5.9) node[above] {$q$};
\node[below left] at (-4.6,0) {$0$};
\draw (-4.67,1.5)--(-4.53,1.5) node[left,xshift=-3pt] {$t_0$};
\draw (0,-0.07)--(0,0.07);
\node[below] at (0,-0.07) {$x_0$};
% The four cones share characteristic edges, each drawn only once.
% Their bases at q=0 are [-4,7.1], [-4,1.5], [-1.5,7.1], [-1.5,1.5].
\draw[gray!60,dashed] (-4,0)--(C12)--(7.1,0);
\draw[gray!60,dashed] (-1.5,0)--(C2);
\draw[gray!60,dashed] (C1)--(1.5,0);
\fill[blue!7] (C0)--(C1)--(C12)--(C2)--cycle;
\draw[thick,blue] (C0)--(C1)--(C12)--(C2)--cycle;
\fill (C0) circle(1.5pt) node[left,xshift=-2pt] {$P_0$};
\fill (C1) circle(1.5pt) node[left] {$P_1$};
\fill (C2) circle(1.5pt) node[right] {$P_2$};
\fill (C12) circle(1.5pt) node[above] {$P_{12}$};
\node at (0.75,3.45) {$\mathcal P_z(h_1,h_2)$};
\node at (0,0.5) {$\Delta(P_0)$};
\node at (-2.4,1.5) {$\Delta(P_1)$};
\node at (3.4,2.1) {$\Delta(P_2)$};
\node[anchor=east] (cone-label) at (-1.1,4.55) {$\Delta(P_{12})$};
\draw[gray!60,thin] (cone-label.east)--(0.05,4.05);
\end{tikzpicture}
\caption{Characteristic cancellation for the mixed null-coordinate increment.  }
\label{fig:characteristic-cancellation}
\end{figure}
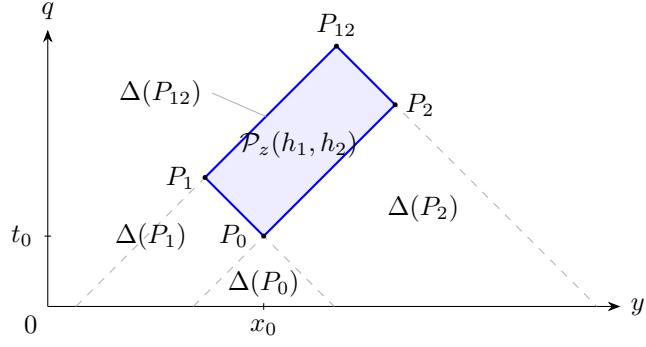

\subsection{{Exact Gaussian variance}}
  \begin{proposition}\label{prop:variance}
For every admissible positive increment $(h_1,h_2)$ at $z$,
the random variable $D_{h_1,h_2}V(z)$ is centered Gaussian
with variance
\begin{equation}\label{eq:variance-main}
\E|D_{h_1,h_2}V(z)|^2=\Psi_H(h_1,h_2).
\end{equation}
Consequently, for $1\leq p<\infty$,
$$
\norm{D_{h_1,h_2}V(z)}_{L^p(\Omega)}
=
m_p\Psi_H(h_1,h_2)^{1/2},
$$
where $m_p=(\E|N|^p)^{1/p}$ and $N$ is a standard normal
random variable. Moreover,
$$
\Psi_H(h_1,h_2)\asymp m^{2H}M,
\qquad m=h_1\wedge h_2,\quad M=h_1\vee h_2,
$$
with comparison constants depending only on $H$.
\end{proposition}

\begin{proof} By Lemma~\ref{lem:geometry}, $D_{h_1,h_2}V(z)$ is a
stochastic integral with a deterministic integrand and is
therefore centered Gaussian. 

The stochastic isometry,
\eqref{eq:indicator-norm}, and the section profile
\eqref{eq:section-profile}, illustrated in the right panel of
Figure~\ref{fig:parallelogram}, give
$$
\begin{aligned}
\E|D_{h_1,h_2}V(z)|^2
&=
\frac14\int_{t_0}^{t_0+a+b}|I_q|^{2H}\dd q\\
&=
\frac14\left[
2\int_0^{a_-}(2\rho)^{2H}\dd\rho
+(a_+-a_-)(2a_-)^{2H}
\right]\\
&=
2^{H-5/2}\left[
(M-m)m^{2H}
+\frac{2m^{2H+1}}{2H+1}
\right]\\
&=
\Psi_H(h_1,h_2).
\end{aligned}
$$
The $L^p(\Omega)$ identity follows from Gaussian scaling.

Finally, since $0<m/M\leq1$ and $H<1/2$,
$$
1
\leq
\frac{\Psi_H(h_1,h_2)}
{2^{H-5/2}m^{2H}M}
=
1+\frac{1-2H}{2H+1}\frac mM
\leq
\frac{2}{2H+1},
$$
which proves the uniform comparison.
\end{proof}

\subsection{Proof of Theorem~\ref{thm:main}}

\begin{proof}[Proof of Theorem~\ref{thm:main}]
Fix $T>0$, $2\leq p<\infty$, and an admissible positive
increment $(h_1,h_2)$ at $z$, with
$$m=h_1\wedge h_2,\ \ \ \ M=h_1\vee h_2\leq1.$$
Let $P_0=(t_0,x_0)$,
and write $a=h_1/\sqrt2$ and $b=h_2/\sqrt2$. 
 
\noindent\textit{Step 1: Reduction to a localized energy estimate.}
Define
$$
R_{h_1,h_2}(z)
=
D_{h_1,h_2}v(z)
-
F(v(z))D_{h_1,h_2}V(z)
$$
and, for $q\in(t_0,t_0+a+b)$, set
$$
A_q(y)=F(u(q,y))-F(u(t_0,x_0)).
$$
By~\eqref{eq:mixed-remainder-representation},
$$
R_{h_1,h_2}(z)
=
\frac12
\int_{t_0}^{t_0+a+b}\int_\R
\1_{I_q}(y)A_q(y)\,W(\dd q,\dd y),
$$
where $I_q$ is the spatial section of
$\mathcal P_z(h_1,h_2)$ at time $q$.
The stochastic integral estimate~\eqref{eq:BDG} therefore gives
$$
\norm{R_{h_1,h_2}(z)}_{L^p(\Omega)}^2
\leq
C\int_{t_0}^{t_0+a+b}
\energyp{p}{\1_{I_q}A_q}\dd q.
$$
It remains to control the energy of the zero-extended
coefficient difference on each section.

\noindent\textit{Step 2: Localized fractional-energy bound.}
The characteristic geometry gives, for
$(q,y)\in\mathcal P_z(h_1,h_2)$,
$$
0<q-t_0<a+b,
\qquad
|y-x_0|\leq\max\{a,b\},
\qquad
|I_q|\leq\sqrt2m.
$$
In particular, $$|q-t_0|+|y-x_0|\leq3M/\sqrt2.$$
Proposition~\ref{prop:regularity} and the Lipschitz continuity
of $F$ consequently yield
$$
\sup_{y\in I_q}\norm{A_q(y)}_{L^p(\Omega)}
\leq
C\sup_{y\in I_q}
\big(|q-t_0|^H+|y-x_0|^H\big)
\leq CM^H.
$$
Moreover, the frozen value cancels in spatial differences, so
for $y,y'\in I_q$,
$$
\begin{aligned}
\norm{A_q(y)-A_q(y')}_{L^p(\Omega)}
&=
\norm{F(u(q,y))-F(u(q,y'))}_{L^p(\Omega)}\\
&\leq C|y-y'|^H.
\end{aligned}
$$
Since $H>1/4$, Lemma~\ref{lem:localized-energy} applies
with $\gamma=H$, $M_0\leq CM^H$, and $M_1\leq C$.
It follows that
$$
\begin{aligned}
\energyp{p}{\1_{I_q}A_q}
&\leq
C\left(
M^{2H}|I_q|^{2H}+|I_q|^{4H}
\right)\\
&\leq
CM^{2H}|I_q|^{2H},
\end{aligned}
$$
where the last inequality uses
$|I_q|\leq\sqrt2m\leq\sqrt2M$.

\noindent\textit{Step 3: Integration and conclusion.}
Combining Steps~1 and~2 with the variance calculation in
Proposition~\ref{prop:variance}, we obtain
$$
\begin{aligned}
\norm{R_{h_1,h_2}(z)}_{L^p(\Omega)}^2
&\leq
CM^{2H}
\int_{t_0}^{t_0+a+b}|I_q|^{2H}\dd q\\
&=
4CM^{2H}\Psi_H(h_1,h_2).
\end{aligned}
$$
Taking square roots proves~\eqref{eq:main}.
The constants are uniform in the base point and independent
of the aspect ratio.\end{proof}

\section{{Anisotropic quadratic variations and parameter estimation}}\label{sec:statistics}
 \subsection{A weighted Gaussian array estimate}

Fix $L_1,L_2>0$ and an observation rectangle
\begin{equation}\label{eq:observation-domain}
\mathcal D
=
[\tau_*,\tau_*+L_1]\times[\lambda_*,\lambda_*+L_2],
\end{equation}
where
$$
t_*:=\frac{\tau_*+\lambda_*}{\sqrt2}>0,
\qquad
t_*+\frac{L_1+L_2}{\sqrt2}\leq T.
$$
For positive integers $N_1$ and $N_2$, set
$$
h_i=\frac{L_i}{N_i},
\qquad
m=h_1\wedge h_2,
\qquad
M=h_1\vee h_2.
$$
For
$$
z_{ij}
=
(\tau_*+ih_1,\lambda_*+jh_2),
\qquad
0\leq i<N_1,\quad 0\leq j<N_2,
$$
write
$$
t_{ij}
=
t_*+\frac{ih_1+jh_2}{\sqrt2},
\qquad
x_{ij}
=
\frac{\lambda_*+jh_2-\tau_*-ih_1}{\sqrt2}
$$
for the corresponding physical coordinates, and define
$$
Y_{ij}=D_{h_1,h_2}V(z_{ij}),
\qquad
\sigma_h^2=\Psi_H(h_1,h_2).
$$

For a cell $c=(i,j)$, we use the shorthand
$$
z_c=z_{ij},\qquad
t_c=t_{ij},\qquad
x_c=x_{ij},\qquad
Y_c=Y_{ij}.
$$

By Lemma~\ref{lem:geometry}, $Y_{ij}$ is a centered Gaussian
stochastic integral with deterministic kernel
$$
\frac12\1_{\mathcal P_{z_{ij}}(h_1,h_2)},
$$
whose time support is $(t_{ij},t_{ij}+\delta)$, where
$$
\delta=\frac{h_1+h_2}{\sqrt2}.
$$
To separate the random weights from the Gaussian increments,
define
$$
b_k=t_*+k\delta,
\qquad
\mathcal I_k
=
\{(i,j):b_k\leq t_{ij}<b_{k+1}\},
\qquad
k\geq0.
$$
If $c=(i,j)\in\mathcal I_k$, then the time support of $Y_c$
is contained in $(b_k,b_{k+2})$. Hence, by temporal whiteness,
the Gaussian vector $(Y_c:c\in\mathcal I_k)$ is independent of
$\F_{b_k}$ and is $\F_{b_{k+2}}$-measurable.

For a cell $c$, let $I_c(q)$ denote the spatial section of
$\mathcal P_{z_c}(h_1,h_2)$ at time $q$, with
$I_c(q)=\varnothing$ outside its time support.
By Proposition~\ref{prop:variance},
$$
\E Y_c^2=\sigma_h^2.
$$

\begin{proposition}[Weighted Gaussian array estimate]
\label{prop:weighted-array}
Let $p\geq2$, and let $w_c\geq0$ be such that
$w_c$ is $\F_{b_k}$-measurable whenever $c\in\mathcal I_k$ and
$$
\sup_c\norm{w_c}_{L^p(\Omega)}<\infty.
$$
Then,
\begin{equation}\label{eq:weighted-array}
\left\|
\frac{h_1h_2}{\sigma_h^2}
\sum_c
w_c\big(Y_c^2-\sigma_h^2\big)
\right\|_{L^p(\Omega)}
\leq
C_p|\mathcal D|^{1/2}
\left(
\sup_c\norm{w_c}_{L^p(\Omega)}
\right)
\sqrt{h_1h_2},
\end{equation}
where the sum runs over all $N_1N_2$ grid cells.
In particular, the constant is independent of the aspect ratio
$h_1/h_2$.
\end{proposition}

\begin{proof}
We first establish the covariance bound
\begin{equation}\label{eq:covariance-row-bound}
\sup_c\sum_d| \E[Y_cY_d]|^2\leq2\sigma_h^4,
\end{equation}
where the sum is over all grid cells.

Let
$$ \mathcal C_{ij}
=
(\tau_*+ih_1,\tau_*+(i+1)h_1)
\times
(\lambda_*+jh_2,\lambda_*+(j+1)h_2). 
$$
be the null-coordinate cell associated with $c=(i,j)$.
The cells $\mathcal C_{ij}$ form a partition of $\mathcal D$
up to their boundaries. Under the linear change of variables
$$
(q,y)
=
\left(
\frac{\tau+\lambda}{\sqrt2},
\frac{\lambda-\tau}{\sqrt2}
\right),
$$
their images are precisely the characteristic parallelograms
associated with the increments $Y_c$. Hence these parallelograms
have pairwise disjoint interiors.

For $c\neq d$, their spatial sections $I_c(q)$ and $I_d(q)$ are
therefore disjoint up to endpoints for almost every $q$. For two
disjoint intervals $I,J\subset\R$, the definition of
$\mathcal H_H$ gives
$$
\left\langle\1_I,\1_J\right\rangle_{\Hh_H}
=
-2a_H
\int_I\int_J
|y-z|^{2H-2}\dd y\dd z
\leq0,
$$
because $a_H=H(1/2-H)>0$. By the stochastic isometry,
$$
\E[Y_cY_d]
=
\frac14
\int_0^T
\left\langle
\1_{I_c(q)},\1_{I_d(q)}
\right\rangle_{\Hh_H}
\dd q,
$$
and consequently
\begin{equation*} 
\E[Y_cY_d]\leq0,
\qquad c\neq d.
\end{equation*}

Next, let $\mathcal P_{\mathcal D}$ denote the image of
$\mathcal D$ under the above linear transformation, and let
$J(q)$ be its spatial section at time $q$. Since the grid cells
partition $\mathcal D$,
$$
\sum_d\1_{I_d(q)}
=
\1_{J(q)}
\qquad\text{a.e.}
$$
Moreover, whenever $I_c(q)$ is nonempty,
$I_c(q)\subseteq J(q)$.

If $I=(a,b)\subseteq J=(A,B)$, then the covariance identity
associated with~\eqref{eq:indicator-norm} gives 
$$
\begin{aligned}
\left\langle\1_I,\1_J\right\rangle_{\Hh_H}
=
\frac12\Big[
&(b-A)^{2H}-(a-A)^{2H}\\
&+(B-a)^{2H}-(B-b)^{2H}
\Big]
\geq0.
\end{aligned}
$$
Therefore,
$$
\begin{aligned}
\sum_d \E[Y_cY_d]
&=
\frac14
\int_0^T
\left\langle
\1_{I_c(q)},
\sum_d\1_{I_d(q)}
\right\rangle_{\Hh_H}
\dd q\\
&=
\frac14
\int_0^T
\left\langle
\1_{I_c(q)},\1_{J(q)}
\right\rangle_{\Hh_H}
\dd q
\geq0.
\end{aligned}
$$

Using 
Proposition~\ref{prop:variance} and 
$\E[Y_cY_d]\leq0$ for $c\neq d$, we have
$$
\begin{aligned}
\sum_d|\E[Y_cY_d]|
&=
\E\left[Y_c^2\right]-\sum_{d\neq c}\E[Y_cY_d]\\
&=
2\sigma_h^2-\sum_d\E[Y_cY_d]
\leq
2\sigma_h^2.
\end{aligned}
$$
By the Cauchy--Schwarz  inequality,
$$
|\E[Y_cY_d]|
\leq
\big(\E Y_c^2\,\E Y_d^2\big)^{1/2}
=
\sigma_h^2.
$$
Hence
$$
\sum_d|\E[Y_cY_d]|^2
\leq
\sigma_h^2\sum_d|\E[Y_cY_d]|
\leq
2\sigma_h^4,
$$
which proves~\eqref{eq:covariance-row-bound}.

For each block, set
$$
A_k
=
\sum_{c\in\mathcal I_k}
w_c\big(Y_c^2-\sigma_h^2\big).
$$ 

Conditional on $\F_{b_k}$, the coefficients $w_c$ are
deterministic and $A_k$ belongs to the second Gaussian chaos
generated by $(Y_c:c\in\mathcal I_k)$. Moreover,
$$
\E[A_k^2\mid\F_{b_k}]
=
2\sum_{c,d\in\mathcal I_k}
w_cw_d\, \E[Y_cY_d] ^2.
$$
Hence Gaussian hypercontractivity yields
$$
\left(
\E[|A_k|^p\mid\F_{b_k}]
\right)^{2/p}
\leq
C_p
\sum_{c,d\in\mathcal I_k}
w_cw_d\, \E[Y_cY_d]^2.
$$
Taking $L^{p/2}(\Omega)$ norms and using H\"older's inequality
together with~\eqref{eq:covariance-row-bound}, we obtain
$$
\begin{aligned}
\norm{A_k}_{L^p(\Omega)}^2
&\leq
C_p
\sum_{c,d\in\mathcal I_k}
\norm{w_cw_d}_{L^{p/2}(\Omega)}
\,\E[Y_cY_d]^2\\
&\leq
C_p
\left(
\sup_c\norm{w_c}_{L^p(\Omega)}
\right)^2
|\mathcal I_k|\sigma_h^4.
\end{aligned}
$$

By the block construction,
$A_k$ is $\F_{b_{k+2}}$-measurable and
$$
\E[A_k\mid\F_{b_k}]=0.
$$
Thus, for each $\varepsilon\in\{0,1\}$,
$(A_{2n+\varepsilon})_{n\geq0}$ is a martingale-difference
sequence with respect to
$$
\big(
\F_{b_{2n+\varepsilon+2}}
\big)_{n\geq-1}.
$$
Applying the discrete Burkholder inequality to the two parity
subsequences and then Minkowski's inequality gives
$$
\begin{aligned}
\norm{\sum_kA_k}_{L^p(\Omega)}
&\leq
C_p
\left(
\sum_k\norm{A_k}_{L^p(\Omega)}^2
\right)^{1/2}\\
&\leq
C_p
\left(
\sup_c\norm{w_c}_{L^p(\Omega)}
\right)
\sigma_h^2
\left(
\sum_k|\mathcal I_k|
\right)^{1/2}\\
&=
C_p
\left(
\sup_c\norm{w_c}_{L^p(\Omega)}
\right)
\sigma_h^2\sqrt{N_1N_2}.
\end{aligned}
$$
Since
$$
N_1N_2=\frac{|\mathcal D|}{h_1h_2},
$$
multiplication by $h_1h_2/\sigma_h^2$ proves
\eqref{eq:weighted-array}.
\end{proof}

 \subsection{Quadratic variation of the nonlinear solution}

Define the rectangular quadratic-variation statistic
\begin{equation}\label{eq:Q-definition}
Q_{N_1,N_2}(v)
=
\frac{h_1h_2}{\Psi_H(h_1,h_2)}
\sum_{i=0}^{N_1-1}\sum_{j=0}^{N_2-1}
\big[D_{h_1,h_2}v(z_{ij})\big]^2.
\end{equation}
We also introduce the Riemann sum
\begin{equation}\label{eq:S-definition}
S_{N_1,N_2}(v)
=
h_1h_2
\sum_{i=0}^{N_1-1}\sum_{j=0}^{N_2-1}
F^2(v(z_{ij}))
\end{equation}
and its limiting integral
$$
\mathcal A
=
\int_{\mathcal D}
F^2(v(\tau,\lambda))
\dd\tau\dd\lambda.
$$
   \begin{theorem}\label{thm:quadratic-variation}
Assume Condition~\ref{cond:main} and let $\mathcal D$ satisfy
\eqref{eq:observation-domain}. Then, for every
$1\leq p<\infty$ and $M\leq1$,
\begin{equation}\label{eq:quadratic-limit}
\norm{Q_{N_1,N_2}(v)-\mathcal A}_{L^p(\Omega)}
+
\norm{S_{N_1,N_2}(v)-\mathcal A}_{L^p(\Omega)}
\leq
C M^H,
\end{equation}
where $C=C(T,p,H,F,I_0,\mathcal D)$ is independent of
$N_1$, $N_2$, and the aspect ratio $h_1/h_2$.
Consequently, as $\min(N_1,N_2)\to\infty$,
$$
Q_{N_1,N_2}(v)\longrightarrow\mathcal A,
\qquad
S_{N_1,N_2}(v)\longrightarrow\mathcal A
$$
in $L^p(\Omega)$, and hence in probability.
\end{theorem}

 \begin{proof}
It suffices to prove~\eqref{eq:quadratic-limit} for $p\geq2$.
For each cell $c$, set
$$
f_c=F(v(z_c)),
\qquad
R_c=D_{h_1,h_2}v(z_c)-f_cY_c.
$$ 
Decompose
$$
\begin{aligned}
Q_{N_1,N_2}(v)-\mathcal A
={}&
\frac{h_1h_2}{\sigma_h^2}
\sum_c
\Big[
\big(D_{h_1,h_2}v(z_c)\big)^2-f_c^2Y_c^2
\Big]\\
&+
\frac{h_1h_2}{\sigma_h^2}
\sum_c
f_c^2\big(Y_c^2-\sigma_h^2\big)
+
S_{N_1,N_2}(v)-\mathcal A\\
={}&E_1+E_2+E_3.
\end{aligned}
$$

\noindent\textit{(1) Local-linearization term $E_1$.}
By Theorem~\ref{thm:main}, H\"older's inequality, and the
uniform moment bounds,
$$
\norm{R_c}_{L^{2p}(\Omega)}
\leq
CM^H\sigma_h,
\qquad
\norm{f_cY_c}_{L^{2p}(\Omega)}
\leq
C\sigma_h.
$$
Since
$$
D_{h_1,h_2}v(z_c)=f_cY_c+R_c,
$$
the difference-of-squares identity and $M\leq1$ give
$$
\begin{aligned}
 \norm{
\big(D_{h_1,h_2}v(z_c)\big)^2-f_c^2Y_c^2
}_{L^p(\Omega)} 
& \leq
\norm{R_c}_{L^{2p}(\Omega)}
\left(
2\norm{f_cY_c}_{L^{2p}(\Omega)}
+\norm{R_c}_{L^{2p}(\Omega)}
\right)\\
& \leq
CM^H\sigma_h^2.
\end{aligned}
$$
Therefore,
$$
\norm{E_1}_{L^p(\Omega)}
\leq
\frac{h_1h_2}{\sigma_h^2}
\sum_c CM^H\sigma_h^2
\leq
CM^H.
$$

\noindent\textit{(2) Weighted Gaussian term $E_2$.}
The weight $f_c^2=F^2(v(z_c))$ is generally only
$\F_{t_c}$-measurable, whereas
Proposition~\ref{prop:weighted-array} requires
$\F_{b_k}$-measurable weights on each block $\mathcal I_k$.
We therefore freeze the coefficient once more at the block
time $b_k$.

For $c\in\mathcal I_k$, set
$$
\bar f_c=F(u(b_k,x_c)).
$$
Then $\bar f_c^2$ is $\F_{b_k}$-measurable, and
$$
\begin{aligned}
E_2
= 
\frac{h_1h_2}{\sigma_h^2}
\sum_c
(f_c^2-\bar f_c^2)(Y_c^2-\sigma_h^2) 
 +
\frac{h_1h_2}{\sigma_h^2}
\sum_c
\bar f_c^2(Y_c^2-\sigma_h^2).
\end{aligned}
$$
Since $0\leq t_c-b_k<\delta\leq\sqrt2M$,
Proposition~\ref{prop:regularity}, the uniform moment bounds,
and Proposition~\ref{prop:variance} imply
$$
\norm{f_c^2-\bar f_c^2}_{L^{2p}(\Omega)}
\leq
CM^H,
$$
$$
\norm{Y_c^2-\sigma_h^2}_{L^{2p}(\Omega)}
\leq
C_p\sigma_h^2,
\qquad
\sup_c\norm{\bar f_c^2}_{L^p(\Omega)}
\leq
C.
$$
Hence H\"older's inequality gives
$$
\left\|
\frac{h_1h_2}{\sigma_h^2}
\sum_c
(f_c^2-\bar f_c^2)(Y_c^2-\sigma_h^2)
\right\|_{L^p(\Omega)}
\leq
CM^H,
$$
while Proposition~\ref{prop:weighted-array}, applied with
$w_c=\bar f_c^2$, yields
$$
\left\|
\frac{h_1h_2}{\sigma_h^2}
\sum_c
\bar f_c^2(Y_c^2-\sigma_h^2)
\right\|_{L^p(\Omega)}
\leq
C\sqrt{h_1h_2}.
$$
Consequently,
$$
\norm{E_2}_{L^p(\Omega)}
\leq
C\big(M^H+\sqrt{h_1h_2}\big)
\leq
CM^H,
$$
because $\sqrt{h_1h_2}\leq M\leq M^H$ for $M\leq1$.

\noindent\textit{(3) Riemann-sum term $E_3$.}
For $z$ in the cell based at $z_c$, the corresponding physical
coordinates differ by at most $CM$. Hence
Proposition~\ref{prop:regularity}, the Lipschitz continuity of
$F$, and the uniform moment bounds give
$$
\sup_{z\in\mathrm{cell}\ c}
\norm{F^2(v(z))-f_c^2}_{L^p(\Omega)}
\leq
CM^H.
$$
Since $E_3=S_{N_1,N_2}(v)-\mathcal A$, Minkowski's inequality
yields
$$
\begin{aligned}
\norm{E_3}_{L^p(\Omega)}
&\leq
\sum_c
\int_{\mathrm{cell}\ c}
\norm{F^2(v(z))-f_c^2}_{L^p(\Omega)}
\dd z\\
&\leq
CM^H.
\end{aligned}
$$

   The estimate for $E_3$ proves the asserted bound for
$S_{N_1,N_2}(v)-\mathcal A$. Combining the estimates for
$E_1$, $E_2$, and $E_3$ proves~\eqref{eq:quadratic-limit}
for $p\geq2$. The case $1\leq p<2$ follows from monotonicity
of $L^p(\Omega)$ norms.

Since
$$
M
\leq
\frac{\max\{L_1,L_2\}}{\min\{N_1,N_2\}},
$$
we have $M\to0$ whenever $\min(N_1,N_2)\to\infty$. 
   
      The proof is complete. 
\end{proof}

\subsection{A diffusion-multiplier estimator}

We now apply the quadratic-variation limit to estimate a
multiplicative diffusion parameter. Consider the parametric family
$$
\partial_t^2u_\theta(t,x)
=
\partial_x^2u_\theta(t,x)
+
\theta F(u_\theta(t,x))\dot W(t,x),
\qquad
\theta>0,
$$
with the same initial data as in~\eqref{eq:SWE}, and let
$v_\theta$ denote the corresponding solution in null coordinates.
Throughout this subsection, the Hurst parameter $H$ and the
function $F$ are assumed to be known, whereas $\theta$ is the
unknown parameter to be estimated from discrete observations of
$v_\theta$ on $\mathcal D$.

Let $Q_{N_1,N_2}(v_\theta)$ and
$S_{N_1,N_2}(v_\theta)$ be defined by
\eqref{eq:Q-definition} and~\eqref{eq:S-definition},
respectively, with $v$ replaced by $v_\theta$, and set
$$
\mathcal A_\theta
=
\int_{\mathcal D}F^2(v_\theta(z))\dd z.
$$

For each fixed $\theta>0$, the equation with diffusion
coefficient $\theta F$ satisfies Condition~\ref{cond:main}.
Hence, as $\min(N_1,N_2)\to\infty$,
Theorem~\ref{thm:quadratic-variation}, applied with
$\theta F$ in place of $F$, gives, for every $1\leq p<\infty$,
\begin{equation}\label{eq:theta-Q-limit}
Q_{N_1,N_2}(v_\theta)
\longrightarrow
\theta^2\mathcal A_\theta
\qquad\text{in }L^p(\Omega),
\end{equation}
 The corresponding Riemann sum in
Theorem~\ref{thm:quadratic-variation} is
$$
h_1h_2
\sum_{i=0}^{N_1-1}\sum_{j=0}^{N_2-1}
\theta^2F^2(v_\theta(z_{ij}))
=
\theta^2S_{N_1,N_2}(v_\theta).
$$
Hence, 
\begin{equation}\label{eq:theta-S-limit}
S_{N_1,N_2}(v_\theta)
\longrightarrow
\mathcal A_\theta
\qquad\text{in }L^p(\Omega).
\end{equation}
In particular, both convergences hold in probability.

 To ensure identifiability, we impose the nondegeneracy condition
in the following proposition.  It rules out the degenerate case where $F(v_\theta)$ vanishes
identically on the observation domain, in which case the
diffusion parameter cannot be identified from the observations.
\begin{proposition}[Consistency]\label{prop:estimator}
Assume in addition that
$\mathcal A_\theta>0$ almost surely.

Define
$$
\widehat\theta_{N_1,N_2}
=
\begin{cases}
\displaystyle
\left(
\frac{Q_{N_1,N_2}(v_\theta)}
     {S_{N_1,N_2}(v_\theta)}
\right)^{1/2},
&
S_{N_1,N_2}(v_\theta)>0,\\[8pt]
0,
&
S_{N_1,N_2}(v_\theta)=0.
\end{cases}
$$
Then
$$
\widehat\theta_{N_1,N_2}
\xrightarrow{\mathbb P} \theta, \ \ \ \text{as } 
 \min(N_1,N_2)\to\infty.
$$ 
\end{proposition}
 
 \begin{proof}
By Theorem~\ref{thm:quadratic-variation},
$$
\big(
Q_{N_1,N_2}(v_\theta),
S_{N_1,N_2}(v_\theta)
\big)
\xrightarrow{\mathbb P}
\big(
\theta^2\mathcal A_\theta,
\mathcal A_\theta
\big).
$$
Define $g:[0,\infty)^2\to[0,\infty)$ by
$$
g(x,y)
=
\begin{cases}
\sqrt{x/y},&y>0,\\
0,&y=0.
\end{cases}
$$
Then
$$
\widehat\theta_{N_1,N_2}
=
g\big(
Q_{N_1,N_2}(v_\theta),
S_{N_1,N_2}(v_\theta)
\big).
$$
Since $\mathcal A_\theta>0$ almost surely, the random point
$(\theta^2\mathcal A_\theta,\mathcal A_\theta)$ belongs almost
surely to the continuity set of $g$. Hence the continuous
mapping theorem gives
$$
\widehat\theta_{N_1,N_2}
\xrightarrow{\mathbb P}
g(\theta^2\mathcal A_\theta,\mathcal A_\theta)
=
\theta.
$$ 

The proof is complete.
\end{proof}

\begin{remark}
The nondegeneracy condition in Proposition~\ref{prop:estimator}
is necessary in general. For instance, if
$u_0\equiv v_0\equiv0$, then $F(0)=0$ and uniqueness imply
$u_\theta\equiv0$, so that $\mathcal A_\theta=0$ and $\theta$
is not identifiable.

A simple sufficient condition is the existence of a deterministic
point $z_0\in\mathcal D$ such that
$$
\mathbb P\big(F(v_\theta(z_0))\neq0\big)=1.
$$
Indeed, continuity of $v_\theta$ and $F$ then implies
$\mathcal A_\theta>0$ almost surely.
\end{remark}

 The preceding results yield a consistent estimator of the parameter $\theta$ based on discrete observations of the solution $u_\theta$ over a space--time grid. In this sense, they provide a partial answer to the open problem raised in \cite[Remark~5.2]{AGT2022}. In contrast to the approach in \cite[Section~4]{TZ2025}, which is based on discrete temporal observations at a fixed spatial location, the present estimator is constructed from observations over a two-dimensional space--time grid and thus incorporates both the temporal and spatial dependence structures of the solution.

\subsection{Numerical illustration}\label{subsec:numerical}

We conclude with a numerical illustration of the estimator in
Proposition~\ref{prop:estimator}.  We take
$$
H=0.35,\qquad \theta=1,\qquad F(x)=\tanh(x),
$$
and use the initial data
$$
u_0(x)=e^{-x^2},\qquad v_0(x)=0.
$$
Thus $F(0)=0$, while the nonzero initial displacement avoids the
trivial solution arising from zero initial data.  We set
$L_1=L_2=0.5$ and choose
$\tau_*=\lambda_*=t_*/\sqrt2$ with $t_*\approx0.4005$.

To simulate the equation, we discretize its mild formulation on a
fine physical grid $t_n=n\Delta t$, $x_j=j\Delta x$ with
$\Delta t=\Delta x$. With $L:=0.5$, we take the reference
resolution $N_{\mathrm{ref}}=128$, so that
$$
\Delta t=\Delta x
=
\frac{L}{\sqrt{2}N_{\mathrm{ref}}}
\approx 2.7621\times10^{-3}.
$$
The computation is performed on the physical spatial interval
$[-1.475,1.475]$ up to time $T\approx1.1076$.
The resulting approximation is
\begin{equation}\label{eq:numerical-mild}
 u_{n,j}^{\Delta}
 =I_0(t_n,x_j)
 +\frac{\theta}{2}
 \sum_{r=0}^{n-1}
 \sum_{|x_j-x_k|<t_n-t_r}
 F(u_{r,k}^{\Delta})\,\Delta W_{r,k}.
\end{equation}
The spatially rough noise increments are generated by circulant
embedding.  More precisely,
$$
\operatorname{Cov}(\Delta W_{r,k},\Delta W_{s,\ell})
=
\mathbf 1_{\{r=s\}}\,\Delta t\,(\Delta x)^{2H}
\gamma_H(k-\ell),
$$
where
$$
\gamma_H(q)
=\frac12\Big(
 |q+1|^{2H}+|q-1|^{2H}-2|q|^{2H}
 \Big).
$$
For each simulated path, the values of $u^\Delta$ are sampled on the
null-coordinate grid, and $Q_{N_1,N_2}$, $S_{N_1,N_2}$, and
$\widehat\theta_{N_1,N_2}$ are then computed exactly as in
\eqref{eq:Q-definition}, \eqref{eq:S-definition}, and
Proposition~\ref{prop:estimator}.
  
  We use $250$ Monte Carlo replications with random seed $20260927$
and consider both balanced meshes $N_2=N_1$ and anisotropic meshes
$N_2=2N_1$. 
The results are reported in Table~\ref{tab:theta-simulation}.
 For both mesh designs, the bias, standard deviation, and root mean
squared error decrease with mesh refinement, indicating improved
finite-sample performance.

\begin{table}[!htbp]
\centering
\caption{Monte Carlo performance of $\widehat\theta_{N_1,N_2}$ for
$H=0.35$ and $\theta=1$, based on $250$ replications.}
\label{tab:theta-simulation}
\begin{tabular}{ccrrrr}
\toprule
$N_1$ & $N_2$ & Mean & Bias & SD & RMSE\\
\midrule
8  & 8  & 0.9568 & -0.0432 & 0.1013 & 0.1100\\
16 & 16 & 0.9788 & -0.0212 & 0.0560 & 0.0598\\
32 & 32 & 0.9928 & -0.0072 & 0.0294 & 0.0302\\
\midrule
8  & 16 & 0.9652 & -0.0348 & 0.0808 & 0.0878\\
16 & 32 & 0.9866 & -0.0134 & 0.0400 & 0.0421\\
32 & 64 & 0.9974 & -0.0026 & 0.0203 & 0.0205\\
\bottomrule
\end{tabular}
\end{table}

Figure~\ref{fig:theta-simulation} complements Table~\ref{tab:theta-simulation}.
The estimator becomes increasingly concentrated around the true value
$\theta=1$ as the mesh is refined. The anisotropic experiment further
illustrates its stability when the two null-coordinate resolutions differ,
in agreement with the aspect-ratio-free result of
Theorem~\ref{thm:quadratic-variation}.

\begin{figure}[!htbp]
\centering
\includegraphics[width=0.96\textwidth]{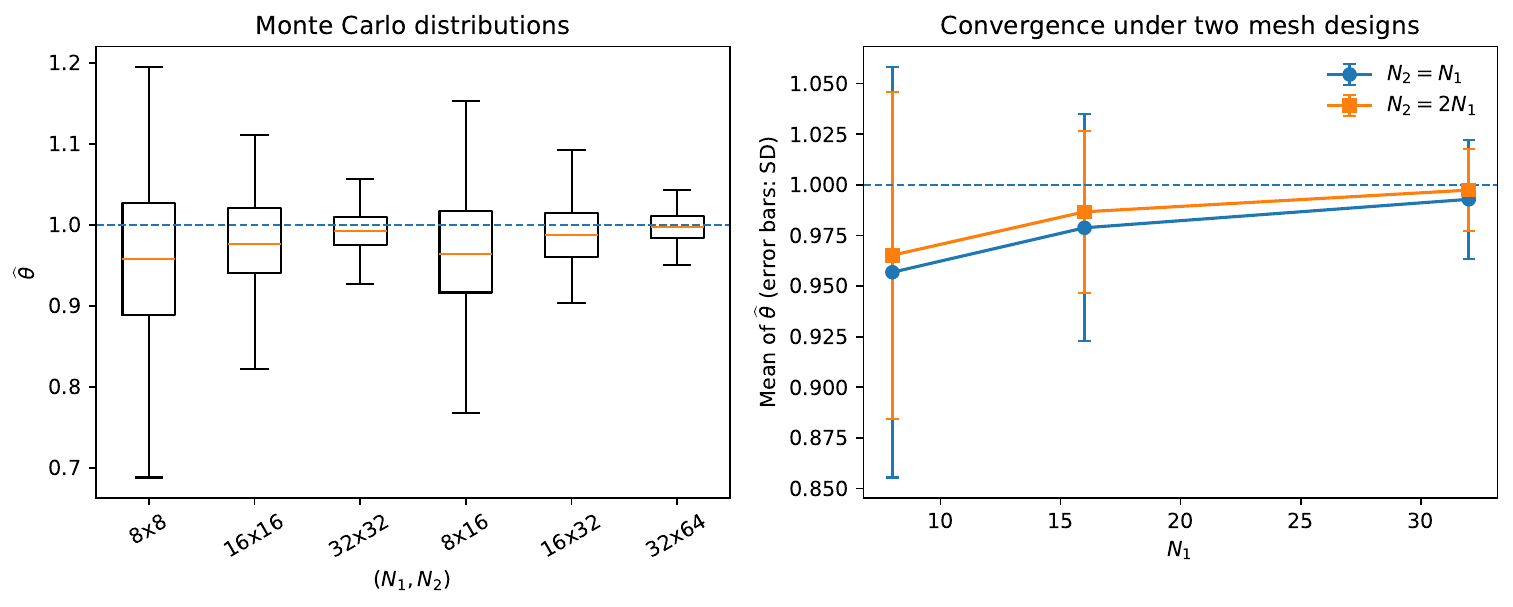}
 \caption{Monte Carlo behavior of the diffusion-multiplier estimator.
Left: empirical distributions of $\widehat\theta_{N_1,N_2}$ for
balanced and anisotropic grids. Right: empirical means with
one-standard-deviation error bars for $N_2=N_1$ and $N_2=2N_1$.
The horizontal dashed line represents the true value $\theta=1$.} 
\label{fig:theta-simulation}
\end{figure}

\end{document}